\documentclass[11pt,fleqn]{article}
\usepackage[utf8]{inputenc}
\usepackage{amsmath,amsfonts,amsthm,amssymb}

\newcommand{\R}{\mathbb R}
\newcommand{\C}{\mathbb C}
\newcommand{\Z}{\mathbb Z}
\newcommand{\D}{\mathbb D}
\newcommand{\diag}{\mathop {\rm diag}\nolimits}
\newcommand{\tr}{\mathop {\rm tr}\nolimits}

\newtheorem{theorem}{Theorem}
\newtheorem{lemma}[theorem]{Lemma}

\begin{document}
\date {\today}
\title{A New Proof that the Numerical Range is a Complete 2-Spectral Set
for Weighted Shift Matrices}

\author{Michel Crouzeix\footnote{Univ.\,Rennes, CNRS, IRMAR\,-\,UMR\,6625, F-35000 Rennes, France.
email: michel.crouzeix@univ-rennes.fr},  Anne Greenbaum\footnote{University of Washington,
Applied Math Dept., Box 353925, Seattle, WA 98195.  email:  greenbau@uw.edu},
}

\maketitle
\begin{abstract}
In this paper we give an alternative proof that the family of matrices studied by Daeshik Choi 
in {\em A proof of Crouzeix's conjecture for a class of matrices},
{ Linear Algebra and its Applications}, {\bf 438}, no.~8 (2013), pp.~3247--3257,
satisfy Crouzeix's conjecture. We also show that they satisfy the completely bounded version
of the conjecture.
\end{abstract}

\section{Introduction}
In this paper we consider the following class of $d\times d$ matrices
\[
M(\alpha_1,\dots,\alpha_d)=P_d\,\diag(\alpha_1,\dots,\alpha_d),\quad \alpha_1,\dots,\alpha_d\in\C.
\]
where $P_d$ is the cyclic permutation $d\times d$ matrix defined by
$p_{d1}=1$, $p_{ij}=1$ if $j{-}i=1$, $p_{ij}=0$ otherwise. 
The matrices of this family are known to have numerical ranges invariant by the rotation centered at $0$ of angle $2\pi /d$ \cite{iss,tsaiwu}
and have been considered by Daeshik Choi \cite{choi} who has shown that they satisfy the conjecture \cite{crzx} $\psi (M)\leq 2$, where
\begin{align*}
\psi (M)=\sup\{\Vert p(M)\Vert \,;\hbox{ $p$ a polynomial, }|p(z)|\leq1\hbox{ in }W(M)\},
\end{align*}
and $W(M)$ is the numerical range of $M$.
The aim of this paper is to give an alternative proof of this result obtained by using a technique of extremal pairs (i.e., left and right singular vectors corresponding to the largest singular value of a matrix) as initiated in \cite{cgl}, see also \cite{sv}. We also show that, for this familly, $\psi _{cb}(M)=\psi (M)\leq 2$
where $\psi _{cb}$ is the complete bound defined by
\begin{align*}
\psi_{cb}(M)=\sup\{\Vert P(M)\Vert\, ; \hbox{ $P$ a polynomial from $\C$ into $\C^{m,m}$, }\|P(z)\|\leq1\hbox{ in }W(M), m\geq1\}.
\end{align*}
A classical way to study these quantities is to use the Riemann transform $\varphi $ from the interior of $W(M)$ onto 
the open unit disk $\D$ satisfying $\varphi (0)=0$ and $\varphi '(0)>0$, and to remark that 
$\psi (M)=\psi _\D(\varphi (M))$ and $\psi _{cb}(M)=\psi _{cb,\D}(\varphi (M))$, where we have defined
\begin{align*}
\psi_{_\D}(A)&=\sup\{\Vert p(A)\Vert\, ; \hbox{ $p$ a polynomial, }|p(z)|\leq1\hbox{ in }\D\},\\
\psi_{cb,\D}(A)&=\sup\{\Vert P(A)\Vert\, ; \hbox{ $P$ a polynomial from $\C$ into $\C^{m,m}$, }\|P(z)\|\leq1\hbox{ in }\D, m\geq1\}.
\end{align*}

\medskip

The organization of the paper is the following. In Section~2, we show that, for this family, we have 
$\psi _\D(M)=\psi _{cb,\D}(M)$ and that $\psi _\D(M)$ is realized by $p(z)=z^k$, with $1\leq k\leq d{-}1$ when $\psi_\D(M)$ is finite. In Section~3, we  give our new proof of $\psi (M)\leq 2$.\medskip

 \section{Estimate in the unit disk.} \label{M}
 
 In this section, in order to lighten the notation and if there is no ambiguity, we will often use $M$ in place of $M(\alpha_1,\dots,\alpha_d)$. Here we will use an important result due to Vern Paulsen \cite{paul} which says that
\begin{align}
\label{pcbd}
\psi_{cb,\D}(A)&=\min\{\Vert X^{-1}\Vert\,\|X\|\, ;\,X\in\C^{d,d},\  \|X^{-1}AX\|\leq1\},
\end{align}
(with the convention $=+\infty$ if no such $X$ exists).

 The following lemma shows that, for the class of matrices $M$, the complete bound, the polynomial bound, and the power bound, are equal.
  
\begin{lemma}\label{lem1} If $|\alpha_1\cdots\alpha_d|>1$, then $\psi_{\D}(M)=\psi_{cb,\D}(M)=+\infty$.
Otherwise, if $|\alpha_1\cdots\alpha_d|\leq1$, then 
\begin{align*}
\psi_{\D}(M)=\psi_{cb,\D}(M)=\max_{0\leq k}\|M^k\|=\max_{0\leq k\leq d{-}1}\|M^k\|.
\end{align*}

\end{lemma}
\begin{proof} First we remark that there exists $\theta \in\R$ such that the matrix $M(\alpha_1,\dots,\alpha_d)$ is unitarily similar to the matrix $e^{i\theta }M([\alpha_1|,\dots,|\alpha_d|)$, thus we have clearly $\psi_\D (M(\alpha_1,\dots,\alpha_d))=\psi_\D (M([\alpha_1|,\dots,|\alpha_d|))$ and $\psi_{cb,\D} (M(\alpha_1,\dots,\alpha_d))=\psi_{cb,\D} (M([\alpha_1|,\dots,|\alpha_d|))$. Therefore we can assume that $\alpha_j\geq 0$ for $j=1,\dots, d$.

The case $\alpha_1\cdots\alpha_d>1$ directly follows from the relation
$M^d= \alpha_1\cdots\alpha_d\,I$. From now on, we assume  
$\alpha_1\cdots\alpha_d\leq1$. We clearly have
\begin{align*}
\psi_{cb,\D}(M)\geq\psi_{\D}(M)\geq \max_{0\leq k}\|M^k\|=\max_{0\leq k\leq d{-}1}\|M^k\|.
\end{align*}
To obtain converse inequalities, we define for $j\in\Z/d\Z$, $\varpi_1(\alpha_j)=\alpha_j$, 
$\varpi_k(\alpha_j)=\alpha_j\cdots \alpha_{j+k-1}$, $x_j=\max(1,\varpi_1(\alpha_j),\dots,\varpi_{d-1}(\alpha_j))$, and $X=\diag(x_1,\dots,x_d)$. Note that for all $j$, $\varpi_d ( \alpha_j ) =\alpha_1\cdots\alpha_d\leq 1$.
whence 
\[
\alpha_j\,x_{j+1}=\max(\varpi_1(\alpha_j),\varpi_2(\alpha_j),\dots,\varpi_d(\alpha_j))\leq x_j.
\]
We have $X^{-1}M(\alpha_1,\dots,\alpha_d)X=M(\beta _1,\dots,\beta _d)$ with $\beta_j=x_{j-1}^{-1}\alpha_jx_{j}$; hence $\beta _j\leq 1$ and thus $\|X^{-1}M(\alpha_1,\dots,\alpha_d)X
\|\leq 1$. which shows that 
\[
\Psi _{cb,\D}(M)\leq \|X^{-1}\|\,\|X\|\leq \|X\|\leq \max_{j,k}\varpi_k(\alpha_j)\leq \max_{k}\|M^k\|\leq 
\Psi _{\D}(M)\leq \Psi _{cb,\D}(M).
\]
\end{proof} 
The previous proof is strongly inspired by the one of Daeshik Choi  \cite[Theorem\,4]{choi}.

\section{The new proof}
We remark that $M^d=\alpha_1\cdots\alpha_d\, I$ and therefore, if $\alpha_1\cdots\alpha_d\neq 0$, the eigenvalues of $M$ are distinct and are $\lambda_j= \lambda_1e^{2i(j-1)\pi /d}$ with $\lambda_1=(\alpha_1\cdots\alpha_d)^{1/d}$.  From the uniqueness of the Riemann mapping $\varphi $ such that $\varphi(0)=0 $, $\varphi '(0)>0$ and
from the invariance of $W(M)$ by rotation of angle $\pi /d$, we deduce that $c:=\varphi (\lambda _1)/\lambda _1=\varphi (\lambda _j)/\lambda _j$ for all $j$, whence $\varphi (M(\alpha_1,\dots,\alpha_d))=c\,M(\alpha_1,\dots,\alpha_d)=M(c\alpha_1,\dots,c\alpha_d)$. If $\alpha_1\cdots\alpha_d=0$, it is known that $W(A)$ is a disk centered in $0$ (c.f. \cite[Lemma\,2]{choi} or \cite{crzx2}) and thus $\varphi (z)=c\,z$ for some $c>0$.
Therefore, in any case, we have $\varphi (M(\alpha_1,\dots,\alpha_d))=c\,M(\alpha_1,\dots,\alpha_d)=M(c\alpha_1,\dots,c\alpha_d)$ and we deduce from Lemma\,\ref{lem1} applied to $\varphi (M)$ that
\[
\textrm{there exists $k\leq d{-}1$ such that \quad}\psi (M)=\psi _{cb}(M)=\|\varphi (M)^k\|.
\]
\begin{theorem}\label{th7}
If $M=P_d\,\diag(\alpha_1,\dots,\alpha_d)$, then we have $\psi(M)=\psi _{cb}(M)\leq 2$. Furthermore,
if $\alpha_1\cdots\alpha_d\neq 0$ then $\psi(M)=\psi _{cb}(M)< 2$.
 \end{theorem}
 \begin{proof} As in Lemma\,1, we may assume that $\alpha_j>0$ for all $j$; this implies that $\lambda _1>0$, that $W(A)$ is symmetric with respect to the real axis and then that $\varphi (\bar z)=\overline{\varphi (z)}$; in particular $\varphi (x)\in\R^+$ if $x\in\R^+$. We set $ f_0(z)=\varphi (z)^k$ whence $\psi (M)=\psi _{cb}(M)=\|f_0(M)\|$.
There exists a unit vector $x_0$ such that $\|f_0(M)\|=\|f_0(M)x_0\|=\psi(M)$ and it is known that the extremal pair $(f_0,x_0)$ satisfies $\langle f_0(M)\,x_0,x_0\rangle=0$ (cf., for instance, \cite[Proposition 2.2]{hol}).  We now denote by $\partial W$ the boundary of $W(M)$ and 
introduce $g_0(z)=\frac{1}{2\pi  i}\int_{\partial W}\overline{f_0(\tau)}\frac {d\tau} {\tau{-}z}$ and set $\beta(z)=z/\varphi (z)$.  From the invariance by rotation, $\beta $ is a holomorphic function of $z^d$ and we have $\beta (x)>0$ for $x\geq 0$, $x\in W(M)$. We have for $z$ interior to $W(M)$
\begin{align*}
g_0(z)&
=\frac{1}{2\pi  i}\int_{\partial W}\frac{1}{f_0(\tau)}\frac {d\tau } {\tau {-}z}=
\frac{1}{2\pi  i}\int_{\partial W}\frac {\beta(\tau)^k}{\tau^k}\frac{d\tau }{\tau {-}z}=\frac{1}{2\pi  i}\int_{\partial W}\frac {\beta(\tau)^k-\beta (0)^k}{\tau ^k}\frac{d\tau}{\tau {-}z}\\&=\frac {\beta(z)^k-\beta (0)^k}{z^k},
\end{align*}
since the function $\frac {\beta(z)^k-\beta (0)^k}{z^k}$ is holomorphic in $z$. The function $h_0(z):=f_0(z)g_0(z)=\frac {\beta(z)^k- \beta (0)^k}{\beta (z)^k}$ is holomorphic in $W(M)$ and only depends on $z^d$. In particular, for the eigenvalues, we have $h_0(\lambda _1)=h_0(\lambda _2)\cdots =h_0(\lambda _d)$, whence $h_0(M)=h_0(\lambda _1)\,I$.
We now introduce
\begin{align*}
S_0(M)=\frac{1}{2\pi \, i}\int_{\partial W}f_0(\tau) ((\tau {-}M)^{-1}d\tau-(\bar\tau {-}M^*)^{-1}d\bar\tau).
\end{align*}
It is well known that $\|S_0(M)\|\leq 2$ (see, for instance, \cite{CP2017}), and we have $f_0(M)=S_0(M)-g_0(M)^*$, then
\begin{align*}
\psi(M)^2&=\langle f_0(M)x_0,f_0(M)x_0\rangle=\langle f_0(M)\,x_0,S_0(M)x_0\rangle-\langle f_0(M)\,x_0,g_0(M)^*x_0\rangle\\
&=\langle f_0(M)\,x_0,S_0(M)x_0\rangle-\langle h_0(M)\,x_0,x_0\rangle\\
&\leq 2\,\psi (M)-h_0(\lambda _1),
\end{align*}
which shows that $\psi (M)\leq 1{+}\sqrt{1{-}h_0(\lambda _1)}$. \smallskip

In order to prove the theorem, it suffices to show that $h_0(\lambda _1) \geq 0$.  Thus it suffices to show that (in $W(M))$ $\beta (x)$ is an increasing function for $x>0$.
We now consider the inverse function $\sigma$ of $\varphi $.
We have $\beta (x)=\frac{x}{\varphi (x)}=\frac{\sigma (\xi )}{\xi }$ with $\xi =\varphi (x)$. Then, for showing that $\beta (x)$ is increasing with $x>0$, it suffices to prove that $\frac{\sigma (x )}{x}$ is increasing on $x\in(0,1]$.

From the invariance by rotation of angle $\frac{2\pi }d$ it is clear that $\frac{\sigma (\xi )}{\xi }$
 is a holomorphic function of $\xi ^d$ in $\D$, and we denote this function by $\gamma $ :  $\gamma (\xi ^d)=\frac{\sigma (\xi )}{\xi }$.  
From the symmetry of $W(A)$ with respect to the real axis, we deduce that $\gamma (\bar\xi )=\overline{\gamma (\xi )}$, whence $\gamma (x)\in\R$ if $x\in [-1,1]$.

It has been proven in \cite[Lemma\,4]{crzx2}, that $|\sigma (e^{it})|$ is a decreasing $C^1$ function of $t$ on the interval $[0,\frac\pi d]$, hence $|\gamma (e^{it})|=|\sigma (e^{it/d})|$ is a decreasing $C^1$ function of $t$ on the interval $[0,\pi]$ and from $\gamma(\bar z )=\overline{\gamma(z)}$, the function $|\gamma (e^{it})|$ is increasing on the interval $(-\pi ,0)$. Now, since $\gamma $ is nonvanishing in $\D$,
we can write $\gamma (x{+}i\,y)=\exp(u(x{+}i\,y){+}i\,v(x{+}i\,y))$ with $u$ and $v$ real valued harmonic functions on $\D$; then 
 $u(\cos t,\sin t)=\log |\gamma (t)|$ is an increasing function of $\cos t\in [-1,1]$. 
 Since the boundary $\partial W(A)$ is at least $C^1$, the functions $\sigma $, $\gamma $ and $u$ are
$C^1$ in $\D$. Thus we have  $\frac{\partial u}{\partial x}\geq 0$ on the boundary of $\D$. From the maximum principle, we then deduce $\frac{\partial u}{\partial x}> 0$ in $\D$ and in particular on the real axis.  This shows that $\gamma (x)=\exp(u(x,0))$ is increasing, whence $\sigma (x)/x=\gamma (x^d)$ is increasing on $(0,1)$.

\end{proof}
{\noindent}
{\bf Remark\,1.} Note that for the class of matrices in Theorem \ref{th7}, we have proved
not only that $\psi (M) \leq 2$ but also the stronger result that
$\| f_0 (M) \| \leq \| S_0 (M) \|$.  This has been observed
numerically for a wide variety of matrices and is reported in \cite{cgl}.\medskip

{\noindent}
{\bf Remark\,2.} 
 If we consider the matrix $M(2\sin\varphi ,2\cos\varphi,0)$, $\varphi \in[0,\pi /2]$; it is easy to prove that $W(M)$ is the closed unit disk, thus from Lemma 1, $\psi (M)=\psi _\D(M)=2\,\max(\sin\varphi ,\cos\varphi ,\sin(2\varphi ))$, which is $<2$ unless $\varphi =0$, $\varphi =\pi /4$ or $\varphi =\pi /2$. This gives a simple example of a matrix with a disk as numerical range and $\psi (D)\neq 2$.

\section{Some comments}

There exist some situations where the proof of $\psi (A)\leq 2$ is easy to obtain. This is the case when we may exhibit a matrix $X$ with condition number $\|X\|\,\|X^{-1}\|\leq 2$ and $\psi (X^{-1}AX)=1$, thus in particular when $X^{-1}AX$ is a diagonal matrix and cond$(X)\leq 2$; this is also known when $W(A)$ is a disk. The Choi family of matrices studied in this paper is another of the few classes of matrices
where a relatively simple  proof of $\psi (A)\leq 2$ has been given. This has been possible thanks to some different facts:\medskip

a) Symmetry properties of $W(M)$. This implies here a natural choice for a Riemann mapping $\varphi $ from $W(M)$ onto the closed unit disk and symmetry for the eigenvalues of $A$. This leads to a very simple formula : $\varphi (M)=c\,M$.\medskip

b) The fact that $\psi (M)$ is realized by the function $\varphi^k$ for some integer $k$ (and not by
a more general Blaschke product of $\varphi $). This implies that (with the notation of Theorem\,2)  $h_0(M)=h_0(\lambda _1)I$ and then $\langle h_0(M)x_0,x_0\rangle =h_0(\lambda _1)$ since $x_0$ is a unit vector.\medskip

c) The fact that, for the inverse function $\sigma $ of $\varphi $, the modulus $|\sigma(e^{it})|$ is decreasing on $(0,\frac\pi d)$, which allows us to prove that $\beta (x)$ is increasing with $x>0$.\medskip

Our result contains in particular a new proof of $\psi (A)\leq 2$ for $2\times2$ matrices.  Indeed,
if $W(A)$ is not a disk, there exist $a_1,a_2 \in\C$ and $U$ unitary such that
$A{-}\frac12\tr(A) I=U^*M(a_1,a_2)U$. In this $2\times2$ case, our proof is very similar to the one in \cite{mmlro}. Note that the previous fact c) replaces their assumption of bi-circularly symmetric domains. 
\medskip

The relation $\varphi (M)=cM$ occurs when $M$ is diagonalisable and $\varphi (\lambda _i)=c\lambda _i$ for all the eigenvalues $\lambda _i$ of $M$. This property was essential 
in the article {\em Crouzeix's conjecture holds for tridiagonal $3 \times 3$ matrices
with elliptic numerical range centered at an eigenvalue} \cite{glakuli}, but in that case, properties b) and c) are not satisfied. For some of these matrices, it was possible to  exhibit a matrix $X$ with condition number $\|X\|\,\|X^{-1}\|\leq 2$ and $\psi (X^{-1}MX)=1$, but for the others more involved algebraic calculations were necessary including the use of Maxima and Mathematica. 
\medskip

The properties a) and c) are satisfied by the $d\times d$ matrices with invariant by rotation numerical ranges (except maybe in the case where $W(M)$ is a disk) but, in general, this is not the case for
property b). However, we think that there exist some subfamilies, different from the Choi family, which
satisfy a), b), and c);  for them the proof of Theorem\,2 works and we have $\psi (M)\leq 2$.
In particular, from our numerical simulations, it seems that this is the case for the following family parametrized by $a\geq 0$
\begin{align*}
A=\begin{pmatrix}
1&a&0 &c\\0 &i &ic &0\\0 &0 &-1 &-a\\0 &0 &0 &-i
\end{pmatrix},\quad \textrm{with \quad}c=\sqrt{2a^2/(2+a^2)},
\end{align*}
but we have not yet obtained a mathematical proof of property b).

\end{document}